\documentclass[12pt]{amsart}
\usepackage{amsmath,amsthm,amscd,amsfonts,amssymb,enumerate}
\usepackage{graphicx}
\usepackage{color}
\usepackage[colorlinks]{hyperref}
\usepackage{amsfonts,amssymb,amscd,amsmath,enumerate,url,verbatim}
 \usepackage[dvips]{epsfig}
 \usepackage[none]{hyphenat}
\usepackage{amsmath,amssymb,amsfonts,enumerate,amsthm}
 \usepackage{amsgen, amstext,amsbsy,amsopn, amsthm, amsfonts,amssymb,amscd,amsmat
 h,euscript,enumerate,url,verbatim,calc,xypic}
 \usepackage{latexsym}
 \usepackage{graphics}
 \usepackage{color}
\headsep=1cm \numberwithin{equation}{section}
\numberwithin{equation}{section}

\numberwithin{equation}{section}
\input xy
\xyoption{all}

\newtheorem{thm}{Theorem}[section]
\newtheorem{cor}[thm]{Corollary}

\newtheorem{lem}[thm]{Lemma}

\newtheorem{defn}[thm]{Definition}

\newtheorem{exam}[thm]{Example}
\newtheorem{rem}[thm]{Remark}

\newcommand{\Ann}{\operatorname{Ann}\,}

\newcommand{\Hom}{\operatorname{Hom}\,}
\newcommand{\Ext}{\operatorname{Ext}\,}

\newcommand{\Max}{\operatorname{Max}\,}

\newcommand{\Ass}{\operatorname{Ass}\,}
\newcommand{\Assh}{\operatorname{Assh}\,}

\newcommand{\Supp}{\operatorname{Supp}\,}

\newcommand{\grad}{\operatorname{grade}\,}
\newcommand{\depth}{\operatorname{depth}\,}
\renewcommand{\dim}{\operatorname{dim}\,}

\newcommand{\Min}{\operatorname{Min}\,}

\newcommand{\gd}{\operatorname{G-dim}\,}

\newcommand{\h}{\operatorname{ht}\,}

\newcommand{\Z}{\mbox{Z}}

\newcommand{\D}{\mbox{D}}

\newcommand{\fa}{\mathfrak{a}}
\newcommand{\fb}{\mathfrak{b}}

\newcommand{\fm}{\mathfrak{m}}
\newcommand{\fp}{\mathfrak{p}}

\newcommand{\fx}{\mathfrak{x}}
\newcommand{\fc}{\mathfrak{c}}

\begin{document}
\bibliographystyle{amsplain}


\title[Linkage of ideals over a module]
 {Linkage of ideals over a module}

\bibliographystyle{amsplain}

     \author[M. jahangiri]{Maryam jahangiri}
     \author[kh. sayyari]{khadijeh sayyari}



\maketitle
$\address{\textrm{Faculty of Mathematical Sciences and Computer,
Kharazmi University, Tehran, Iran.}}$

$$\email{\textrm{jahangiri@khu.ac.ir    ,
std-sayyari@khu.ac.ir}}$$


\begin{abstract} Inspired by the works in linkage theory
of ideals, we define the concept of linkage of ideals over a module.
Several known theorems in linkage theory are improved or recovered by new approaches.
Specially, we make some extensions and generalizations of the basic result of  Peskine
 and Szpiro \cite[prop 1.3]{PS}, namely if $R$ is a Gorenstain local ring, $\fa \neq 0$ (an ideal of $R$)
 and $\fb := 0:_R \fa$ then $\frac{R}{\fa}$ is Cohen-Macaulay if and only if $\frac{R}{\fa}$ is unmixed and $\frac{R}{\fb}$ is Cohen-Macaulay.

\end{abstract}

\bibliographystyle{amsplain}
\section{introduction}
Classically, linkage theory refers to
Halphen (1870) and M. Noether \cite{No}(1882) who worked to classify
space curves. In 1974 the significant work of Peskine and Szpiro \cite{PS} brought
breakthrough to this theory and stated it in the modern algebraic
language; two proper ideals $\fa$ and $\fb$ in a Cohen-Macaulay local ring $R$ is said to be linked if
there is a regular sequence $\underline{x}$ in their intersection such that
$\fa =\underline{x}:_R \fb$ and $\fb = \underline{x} :_R \fa$.

A new progress in the linkage theory is the work  of  Martsinkovsky
and Strooker
 \cite{MS} which established the concept of linkage of modules.

 Let $R$ be a commutative Noetherian ring with $1\neq 0$ and $M$ be a finitely generated $R$-module.
  In this paper, inspired by the works in the ideal case, we present the concept of linkage of ideals over
  a module; let $\fa, \fb$ and $I\subseteq \fa \cap \fb$ be ideals of $R$ such that $I$ is generated by an
  $M$-regular sequence and $\fa M\neq M \neq \fb M$. Then $\fa$ and $\fb$ are said to be linked by $I$ over
   $M$, denoted by $\fa\sim_{(I;M)}\fb$, if $\fb M = IM:_M\fa$ and $\fa M = IM:_M\fb $. This is a generalization of its classical concept, when
    $M = R$. It is citable that, in general case, linkedness of two ideals over $M$ does not imply linkedness
    of them over $R$ and vice versa (see Example \ref{E3}). But, in some special cases, it does (see \ref{E4}, \ref{t2} and \ref{t3}).

 In this paper, we consider the above generalization and study some of its basic facts. The organization of the paper goes as follows.

In Section 2, we study some basic properties of linkage of ideals over a module.

By the above definition of linkage of ideals over a module it is
natural to ask whether linkedness of two ideals over a module
implies linkedness
    of them over the ring and vice versa. Section 3 has considered this question in the case where $R$
is a Cohen-Macaulay local ring with the canonical module $\omega_R$
and has studied ideals which are linked over $\omega_R$. More precisely,
it is shown that if $\fa \sim _{(I;w_R)} \fb$ and $\frac{R}{\fa}$
and $\frac{R}{\fb}$ are unmixed, then $\fa\sim _{(I;R)}\fb$ (Theorem
\ref{t2}). Also,  if $\fa \sim_{(I;R)} \fb$ and $\frac{w_R}{\fa
w_R}$ and $\frac{w_R}{\fb w_R}$ are unmixed, then $\fa \sim
_{(I;w_R)} \fb$ (Theorem \ref{t3}).

The first main theorem in the theory of linkage is the following.

{\it {Theorem A. \cite[1.3]{PS} If $(R,\fm)$ is a Gorenstein local ring and $\fa$
 and $\fb$ are two linked ideals of R, then $R/\fa$ is Cohen-Macaulay if and only if $R/\fb$ is Cohen-Macaulay.}}

Attempts to generalize this theorem lead to several developments in
linkage theory, especially the works by C. Huneke \cite{Hu}, B.
Ulrich \cite{U1} and \cite{U2}. A counterexample given by Peskine
and Szpiro in   \cite{PS}, shows that {\it Theorem A} is no longer
true if the base ring $R$ is Cohen-Macaulay but not Gorenstein.

In Section 4, we state {\it Theorem A} for linked ideals over the canonical module $\omega_R$. Namely, it is shown that if the ideals $\fa$ and $\fb$ are linked over the canonical module $\omega_R$ and the G-dimension of some certain modules are finite then Cohen-Macaulay-ness of $R/\fa$ and of $R/\fb$ are equivalent (Corollary \ref{c8}).

Moreover, let $(R,\fm)$ be a Cohen-Macaulay local ring and $\fa , \fb$ be pure height ideals of $R$ which are linked over the maximal Cohen-Macaulay $R$-module $M$ of finite injective dimension. Then, it is shown that $\gd_R \frac{R}{\fa} < \infty$ and $\frac{R}{\fa}$ is Cohen-Macaulay if and only if  $\gd_R \frac{R}{\fb} < \infty$ and $\frac{R}{\fb}$ is Cohen-Macaulay (Corollary \ref{c9}).

Throughout the paper, $R$ denotes a non-trivial commutative Noetherian ring, $\fa$ and $\fb$ are non-zero proper ideals of $R$ and $M$ will denote a finitely generated $R$-module.


\section{linked ideals over a module}


The goal of this section is to introduce the concept of linkage of
ideals over a module and study some of its basic properties.

\begin{defn}\label{F1}
 \emph{Assume that $\fa M\neq M\neq \fb M$
 and let $I\subseteq \fa \cap \fb$ be an ideal generated by an $M$-regular sequence.
  Then we say that the ideals $\fa$ and $\fb$ are linked by $I$ over $M$,
  denoted $\fa\sim_{(I;M)}\fb$, if $\fb M = IM:_M\fa$ and $\fa M = IM:_M\fb $.
  The ideals $\fa$ and $\fb$ are said to be geometrically linked by $I$ over $M$ if $\fa M \cap \fb M =
  IM$. Also, we say that the ideal $\fa$ is linked over $M$ if there
  exist ideals $\fb$ and $I$ of $R$ such that $\fa\sim_{(I;M)}\fb$.
  $\fa$ is $M$-selflinked by $I$ if $\fa\sim_{(I;M)}\fa$.}
\end{defn}

Note that this definition is a generalization of the concept of
linkage of ideals in \cite{PS}. But, as the Example \ref{E3} shows,
these two concepts do not coincide although, in some cases they do
(e.g. Example \ref{E1}).

\begin{exam}\label{E2}
Let $x_1,... ,x_n$ be an $M$-regular sequence. Then one can see that
\begin{eqnarray*} (x_1,... ,x_n)\sim_{(((x_1)^2,x_2,... ,x_n);M)}(x_1,... ,x_n).
\end{eqnarray*}
In other words, every $M$-regular sequence is $M$-selflinked.
\end{exam}

\begin{lem}\label{l13}
Let $N$ be a finitely generated $R$-module, $x_1,...,x_t \in \fa \cap \fb$ and $I: = (x_1,...,x_t)$. Then $\fa\sim _{(I;M \oplus N)}\fb$ if and only if $\fa\sim _{(I;M)}\fb$ and $\fa\sim _{(I;N)}\fb$.
\end{lem}

\begin{proof}
Note that we can consider the case where $I =0$. Also, it is straightforward to see that $x_1,...,x_t$ is an $M \oplus N$-regular sequence if and only if it is an $M$-regular sequence and $N$-regular sequence.

Assume that $\fa\sim _{(0;M \oplus N)}\fb$. Then, $\fa\fb M = \fa\fb N =0$. Also, if $m \in 0:_M \fb$ then $(m, 0) \in 0:_{M\oplus N} \fb = \fa M \oplus \fa N$. Therefore, $\fa\sim _{(0;M)}\fb$. Similarly, $\fa\sim _{(0; N)}\fb$.

For the converse, one has $0:_{M\oplus N} \fb = (0:_M \fb) \oplus (0:_N \fb) = \fa M \oplus \fa N = \fa (M \oplus N)$.
\end{proof}

\begin{exam}\label{E1}
Let $F$ be a finitely generated free R-module. Then $\fa\sim_{(0;R)}\fb$ iff $\fa\sim_{(0;F)}\fb$.
\end{exam}

The following lemma, among other things, shows that if $\fa$ and
$\fb$ are linked by $I$ over $M$ then the ideal $I$ has to be
generated by a maximal $M$-regular sequence in $\fa\cap\fb.$
\begin{lem} \label{l6}
Let $I$ be a proper ideal of $R$ such that $\fa\sim_{(I;M)}\fb$. Then
\begin{itemize}
   \item [(i)] $\grad_M \fa = \grad_M \fb= \grad_M I$.
   \item [(ii)]  $\Supp \Hom_R (\frac{R}{\fa},\frac{M}{IM})= \Supp \frac{M}{\fa M}$.
   \item [(iii)]   $\Supp \frac{M}{IM}= \Supp \frac{M}{\fa M}\cup \Supp \frac{M}{\fb M}.$
   \end{itemize}
\end{lem}
\begin{proof}
 \begin{itemize}
   \item [(i)] Note that $\fb M\neq IM$. Hence $0:_{\frac{M}{IM}}\fa \neq 0$ and so $I \subseteq \fa \subseteq \Z_R(\frac{M}{IM})$ which implies that $\grad_M I = \grad_M \fa$.
   \item [(ii)] Let $\fp \in \Supp \frac{M}{\fa M}.$ Then, $\fb M_{\fp} \neq IM_{\fp}$. Hence $\fp \in \Supp \frac{\fb M}{IM} = \Supp\Hom_R (\frac{R}{\fa},\frac{M}{IM})$. The converse is clear.

   \item [(iii)] Follows from (ii) and using the following short exact sequence $$0\rightarrow \frac{\fb M}{IM} \rightarrow \frac{M}{IM}\rightarrow \frac{M}{\fb M}\rightarrow 0.$$
   \end{itemize}
\end{proof}
The following Example shows that the concepts of linkage of ideals over $R$ and over $M$ do not coincide.

\begin{exam}\label{E3}
Let $R: = \frac{k[x,y]}{(xy)}$ and $M:= \frac{k[x,y]}{(x)}$ where $k$ is a field. Via the natural homomorphim $ R \rightarrow M$, $M$ is a finitely generated $R$-module. Set $\fa:=(x)$ and $\fb:=(y)$.
 \begin{itemize}
   \item [(i)] As $y$ is an $M$-regular sequence, $\fb\sim _{((y^2);M)}\fb$. Assume that $\fb\sim _{(I;R)}\fb$ for some ideals $I$. Since $\grad_R\fb = 0$, by \ref{l6}, $I=0$. But, $0:_R\fb \neq \fb$ .
   \item [(ii)] $\fa\sim _{(0;R)}\fb$. Assume that $\fa\sim _{(I;M)}\fb$ for some ideals $I$. As $\grad_M\fa = 0$, by \ref{l6}, $I=0$. On the other hand, $0:_M\fa = M$, which is a contradiction.
   \end{itemize}
\end{exam}

\begin{lem}\label{R1}
Let $I$ be a proper ideal of $R$ such that $\fa\sim_{(I;M)}\fb$. Then, $\frac{M}{\fa M}$ can be embedded in finite copies of $\frac{M}{IM}$.
\end{lem}

\begin{proof}
Let $F \rightarrow \frac{R}{I} \rightarrow \frac{R}{\fb} \rightarrow 0 $ be a free resolution of $\frac{R}{\fb}$ as $\frac{R}{I}$-module. Then, applying the functor $(-)^+:= \Hom_{\frac{R}{I}} (- , \frac{M}{IM})$, we get the exact sequence $0 \rightarrow (\frac{R}{\fb})^+ \rightarrow (\frac{R}{I})^+ \stackrel{f}{\rightarrow} F^+$, where $\frac{M}{\fa M}\cong Im (f) \subseteq F^+ \cong \oplus \frac{M}{IM} $.
\end{proof}

There are some relations between ideals which are linked over a module, as the following lemma shows.
\begin{lem}\label{l10}
Let $I$ be an ideal of $R$ such that $\fa\sim_{(I;M)}\fb$. Then the following statements hold.
 \begin{itemize}
   \item [(i)] $\Min \Ass \frac{M}{IM} \subseteq \Min\Ass \frac{M}{\fa M} \cup \Min\Ass \frac{M}{\fb M} $.
   \item [(ii)] If $\Ass \frac{M}{IM} = \Min\Ass \frac{M}{IM}$, then $\Ass \frac{M}{IM} = \Ass \frac{M}{\fa M} \cup \Ass \frac{M}{\fb M} $ and $\Ass \frac{M}{\fa M} = \Ass \frac{M}{IM} \cap V(\fa) $.
   \item [(iii)] If $\fa M \cap \fb M = IM$, then $\Ass \frac{M}{\fa M} = \Ass \frac{M}{IM} \cap V(\fa) $ and $\Ass \frac{M}{\fa M} \cap\Ass \frac{M}{\fb M}= \emptyset.$
   \end{itemize}
\end{lem}

\begin{proof}
\begin{itemize}
   \item [(i)]  Let $ \fp\in \Min \Ass \frac{M}{IM}$. Assume that $ \fp \notin \Min \Ass \frac{M}{\fa M}$. Hence, in view of \ref{R1}, $\Ass \frac{M}{\fa M} \subseteq \Ass \frac{M}{IM}$ and $(\frac{M}{\fa M})_{\fp} = 0$. Therefore, by \ref{l6}(iii), $\fp \in \Min\Ass \frac{M}{\fb M}$.

   \item [(ii)] By $ (i) $, $$\Ass \frac{M}{IM} \subseteq \Min \Ass \frac{M}{\fa M} \cup \Min \Ass \frac{M}{\fb M} \subseteq \Ass \frac{M}{\fa M} \cup \Ass \frac{M}{\fb M}.$$ On the other hand, by \ref{R1}, $ \Ass \frac{M}{\fa M} \cup \Ass \frac{M}{\fb M} \subseteq \Ass \frac{M}{IM}$.

  Now, let $ \fp\in \Ass \frac{M}{IM} \cap V(\fa)$. By the assumption, $\fp \in \Min (\Ann M + I) \cap V(\fa)$. Hence $\fp \in \Min \Ass \frac{M}{\fa M}$. The converse is clear by \ref{R1}.

   \item [(iii)] $\Ass \frac{M}{\fa M} \subseteq \Ass \frac{M}{IM} \cap V(\fa),$ by \ref{R1}. For the converse, let $\fp \in \Ass \frac{M}{IM} \cap V(\fa)$. Hence, $\fp \in \Ass \Hom_R (\frac{R}{\fa},\frac{M}{IM}) = \Ass \frac{\fb M}{IM} $ and there exists $\alpha \in \fb M$ such that $\fp = 0:_R{\alpha + IM}$. Now, it is straight forward to see that $\fp = 0:_R{\alpha + \fa M}$. Therefore, $\fp \in\Ass \frac{M}{\fa M}$.

Now, let $ \fp\in \Ass \frac{M}{\fa M} \cap\Ass \frac{M}{\fb M}= \Ass M \cap V(\fa +\fb).$ This implies that $\grad_M\fa+\fb = \grad_MI.$ On the other hand, $$0:_{\frac{M}{IM}}\fa+\fb= 0:_{\frac{M}{IM}}\fa\cap0:_{\frac{M}{IM}}\fb=0,$$ which implies that $\grad_M\fa+\fb > \grad_MI$ and this is a contradiction. \end{itemize}
\end{proof}
The following lemma studies the Artinianness of $\frac{M}{\fa M}$ where $\fa$ is linked over $M.$ It will be used in the next corollary which consider the Cohen-Macaulayness of $M_{\fp}$ for some prime ideal $\fp$.
\begin{lem}\label{l7}
Let $(R,\fm)$ be a local ring and $I$ be an ideal of $R$ such that $\fa\sim_{(I;M)}\fb$ and $\Ass \frac{M}{IM} = \Min \Ass \frac{M}{IM}$ . Then the following conditions are equivalent:
 \begin{itemize}
   \item [(i)] $\frac{M}{IM}$ is Artinian.
   \item [(ii)]  $\frac{M}{\fa M}$ is Artinian.
   \item [(iii)] $\frac{M}{\fb M}$ is Artinian.
   \end{itemize}
Moreover, $\fa M \cap \fb M \neq IM$  if one of the above equivalent conditions holds.
\end{lem}

\begin{proof}
 $ (i)\rightarrow(ii) $ is clear.

$ (ii)\rightarrow(i)$  Follows from \ref{R1} and the assumption.

$ (ii)\rightarrow(iii)$  Follows from \ref{l6} and the assumption.

 For the end, if $\fa M \cap \fb M = IM$ then $IM = IM:_M(\fa + \fb)$. Hence $0:_{\frac{M}{IM}}(\fa +\fb) = 0 $ and so $ (\fa+ \fb) \nsubseteq \Z_R(\frac{M}{IM})$. This implies that $\grad_M(\fa +\fb)> \grad_M I = \h_M I= \dim M $ and this is a contradiction.
\end{proof}
\begin{cor}\label{c6}
Let $I$ be an ideal of $R$ such that $\fa\sim_{(I;M)}\fb$ and $\Ass \frac{M}{IM} = \Min \Ass\frac{M}{IM}$. Set $t:= \grad_M I$. Then the following statements hold.
 \begin{itemize}
   \item [(i)] $ \h _M \fp = t$ and $M_{\fp}$ is Cohen-Macaulay for every $ \fp\in \Min \Ass \frac{M}{\fa M} $.
   \item [(ii)] $ \h_M \fa = \h_M \fb = \grad_M \fa = \grad_M \fb = t $. Moreover, if $\fa M \cap\fb M \neq IM$ then $\h _M (\fa + \fb )= \grad _M \fa $.
   \item [(iii)] If $\frac{M}{IM}$ is equidimensional then $ \dim M = \h _M \fp + \dim \frac{R}{\fp}$ for every $ \fp\in \Min \Ass \frac{M}{\fa M} $. Therefore $ \dim M = \h _M \fa + \dim \frac{M}{\fa M}$.
   \end{itemize}
\end{cor}

\begin{proof}
 $ (i) $ Let $\fp \in \Min\Ass \frac{M}{\fa M} $. If $\fp \nsupseteq \fb$, then $\fa M_{\fp} = IM_{\fp} : _{M_{\fp}} \fb R_{\fp} = IM_{\fp}$ and so $0 = \dim \frac{M_{\fp}}{\fa M_{\fp}} = \dim \frac{M_{\fp}}{I M_{\fp}}$. Therefore $M_{\fp}$ is Cohen-Macaulay of dimension $t$.

 In case $\fp\supseteq \fb$, $\fa R_{\fp}\sim_{(IR_{\fp};M_{\fp})}\fb R_{\fp}$ and $\Ass\frac{M_{\fp}}{IM_{\fp}} = \Min \Ass \frac{M_{\fp}}{IM_{\fp}}$. Hence, by \ref{l7}, $M_{\fp}$ is Cohen-Macaulay of dimension $t$.

$(ii)$ The first part follows from $ (i) $. Assume that $\fa M \cap \fb M \neq IM$ then, $0:_{\frac{M}{IM}}(\fa +\fb) \neq 0 $ and so $ (\fa+ \fb) \subseteq Z_R(\frac{M}{IM})$. Hence $(\fa+ \fb) \subseteq \fp$ for some $\fp \in Z_R(\frac{M}{IM})$ and so $\fp\in \Min \Ass \frac{M}{IM}$. By (i), $ \h _M (\fa + \fb )= t $.

   $ (iii) $ Let $ \fp\in \Min\Ass \frac{M}{\fa M} \subseteq \Min\Ass \frac{M}{IM} = \Assh \frac{M}{IM}$. Then $ \dim \frac{R}{\fp} = \dim \frac{M}{IM} = \dim M - t$ and, by (i), the result has desired. Also, in view of the fact that $\dim \frac{M}{\fa M} = \dim \frac{R}{\fp}$ for some $\fp \in \Min\Ass \frac{M}{\fa M}$ and $ \h _M \fa = \h _M \fp$, we have $ \dim M = \h _M \fa + \dim \frac{M}{\fa M}$.
\end{proof}

In the rest of this section we study whether linkedness of ideals over a module transfers via homomorphisms.

\begin{rem} \label{r2}
Let $R' \rightarrow R$ be a ring homomorphism and $\fa'$ and $\fb'$ be non-zero proper ideals of $R'$. Then $\fa'\sim _{(0;M)}\fb'$ if and only if $\fa' R \sim_{(0;M)}\fb' R$.
\end{rem}

In the following, we show that linkage of ideals over a module transfers, in some sense, via faithfully flat homomorphisms.

\begin{lem} \label{l12}
Let $R \rightarrow S$ be a faithfully flat ring homomorphism and $\fa\sim _{(0;M)}\fb$. Then, $\fa S \sim_{(0;M\otimes_R S)}\fb S$ up to isomorphism, i.e. $0:_{M\otimes_R S} \fa S \cong \fb (M\otimes_R S)$ and $0:_{M\otimes_R S} \fb S \cong \fa (M\otimes_R S)$. In particular, if $(R , \fm)$ is local then $\fa \hat{R} \sim_{(0;M\otimes_R \hat{R})}\fb \hat{R},$ up to isomorphism.
\end{lem}

\begin{proof}
 Consider the composition of natural isomorphisms
 \begin{eqnarray*}(0:_M \fa ) \otimes _R S \stackrel {\Phi}{\rightarrow} \Hom_R (\frac{R}{\fa}, M)\otimes_R S \stackrel {\Psi}{\rightarrow} \Hom_S (\frac{S}{\fa S}, M\otimes_R S) \stackrel {\Theta}{\rightarrow} 0:_{M\otimes_R S} \fa S\end{eqnarray*}
 and note that this composition is actually the identity map. So we get \begin{equation}\label{e1}
                                                                          (\fb M)\otimes_R S = 0:_{M\otimes_R S} \fa S.
                                                                        \end{equation}
   On the other hand, there are natural isomorphisms
   \begin{eqnarray*}\frac{M\otimes_R S}{\fb (M\otimes_R S)} \cong M\otimes_R S \otimes_R \frac{R}{\fb} \cong \frac{M\otimes_R S}{(\fb M) \otimes_R S}\end{eqnarray*}
   and the following commutative diagrams
         $$\begin{CD}
         &&&&&&&&\\
         \ \ &&&& 0 @>>> \fb (M\otimes_R S) @>>> M\otimes_R S @>>> {M\otimes_R S}/{\fb (M\otimes_R S)} @>>>0&  \\
          &&&&&&  @VV{h}V   @VV{\cong}V @VV{\cong}V \\
         \ \  &&&& 0 @>>> (\fb M)\otimes_R S  @>>> M\otimes_R S
 @>>> \frac{M\otimes_R S}{(\fb M)\otimes_R S} @>>>0 .&\\
        \end{CD}$$\\
         This implies that $h$ is an isomorphism. Hence, by (\ref{e1}), $0:_{M\otimes_R S} \fa S \cong \fb (M\otimes_R S)$.

         Now, the result follows from these facts that $\fa (M\otimes_R S) \neq (M\otimes_R S)\neq\fb (M\otimes_R S)$.
\end{proof}
The following lemma consider a case where linkedness of ideals over $M$ passes over $\frac{M}{xM}$ where $x$ is an $M$-regular element.
\begin{lem}\label{l14}
Let $\fa\sim _{(0;M)}\fb$ and $\Ext^1_R (\frac{R}{\fa},M) = \Ext^1_R (\frac{R}{\fb},M) = 0$. Also, assume that $x \notin \Z(M)$ and $(\fa,x)M \neq M \neq(\fb,x) M$. Then $(\fa,x) \sim _{(0;\frac{M}{xM})}(\fb,x)$, in other words, $\fa\sim _{(0;\frac{M}{xM})}\fb$.
\end{lem}

\begin{proof}
Applying $\Hom (\frac{R}{\fa},-)$ on $$ 0 \rightarrow M \overset{.x}{\rightarrow} M \rightarrow \frac{M}{xM} \rightarrow 0$$ and using the assumptions, we get the following commutative diagrams with exact rows
     $$\begin{CD}
         &&&&&&&&\\
         \ \ &&&& 0 @>>> \Hom _R(\frac{R}{\fa} , M) @>{.x}>> \Hom _R(\frac{R}{\fa} , M) @>>> \Hom _R(\frac{R}{\fa} , \frac{M}{xM}) @>>>0&  \\
          &&&&&&  @VV{\cong}V   @VV{\cong}V \\
         \ \  &&&& 0 @>>> \fb M @>{.x}>> \fb M
 @>>> \frac{\fb M}{x \fb M} @>>>0.&\\
        \end{CD}$$\\
        This implies that $\Hom _R(\frac{R}{\fa} , \frac{M}{xM}) \cong\frac{\fb M}{x \fb M}$. Considering the composition of natural isomorphisms $$0:_{\frac{M}{xM}} \fa \rightarrow \Hom _R(\frac{R}{\fa} ,\frac{M}{xM})\rightarrow\frac{\fb M}{x \fb M}\rightarrow \frac{\fb M + xM}{xM}$$ and using the fact that this composition is actually the identity map, we get $0:_{\frac{M}{xM}} \fa= \frac{\fb M + xM}{xM} = \fb (\frac{M}{xM})$. Now, the result follows using \ref{r2}.
\end{proof}

   \section{linked ideals over a ring and its canonical module}
 Thoroughout this section, we assume that $(R,\fm)$ is a Cohen-Macaulay local ring with the canonical module $w_R$. The main goal of this section is to study the ideals which are linked over the canonical module.

  More precisely, we study whether linkedness of two ideals over $w_R$ implies linkedness over $R$ and vice versa.

  In spite of Example \ref{E3}, there are some cases where linkage of ideals over an $R$-module ends to linkedness of them over $R$ as the following example shows.
 \begin{exam}\label{E4}
Assume that $\fa$ and $\fb$ are radicals, $M$ is faithful and $\fa\sim _{(0;M)}\fb$. Then $\fa\sim _{(0;R)}\fb$. Indeed, by the assumption, $\fa \fb = 0$. Hence, $\fb \subseteq 0 :_R \fa$. On the other hand, if $r \in 0:_R\fa$ then, by the assumption, $rM \subseteq \fb M $. Therefore, in view of \cite[2.1]{M}, there exist $n \in \mathbb{N}$ and $b_1 ,..., b_n \in \fb$ such that $(r^n + r^ {n-1}b_1 + ...+ b_n)M= 0$. This implies that $r \in \sqrt{\fb}=\fb $. Therefore, $0 :_R \fa =\fb$ and $\fa\sim _{(0;R)}\fb$.
\end{exam}

\begin{lem} \label {l3-1}
Let $S$ be a Gorenstain local ring, $\fc$ be a pure heigth ideal of $S$, namely $\h_S\fp = \h_S\fc$ where $\fp\in \Ass_S\frac{S}{\fc}$, and $\h_S \fc= 0$. Then, $0:_S 0:_S \fc = \fc$.
\end{lem}

\begin{proof}
    Assume to the contrary that $0:_S 0:_S\fc \supsetneq \fc$. Then, there exists $\fp \in \Ass_S\frac{0:_S 0:_S\fc}{\fc}$. Therefore, $\fp \in \Ass_S \frac{S}{\fc}$ and $\fp S_{\fp} \in \Ass_{S_{\fp}} \frac{S_{\fp}}{\fc S_{\fp}}$. On the other hand, $S_{\fp}$ is a Gorenstain ring of zero dimension. Hence $0:_{S_{\fp}} 0:_{S_{\fp}}\fc S_{\fp} = \fc S_{\fp}$ and so $(\frac{0:_S 0:_S\fc}{\fc})_{\fp} = 0$, which is a contradiction.
\end{proof}

 In the following theorem, we consider a case where linking over the canonical module implies linking over $R$.
\begin{thm}\label{t2}
Let $I \subseteq \fa \cap \fb$ be an ideal of $R$ which is generated by an $w_R$-regular sequence. Also,
 assume that $Iw_R :_{w_R}\fb = \fa w_R \neq Iw_ R$ and $\frac{R}{\fb}$ is unmixed. Then $\fb = I:_R \fa$.
 In other words, if $\fa \sim _{(I;w_R)} \fb$ and $\frac{R}{\fa}$ and $\frac{R}{\fb}$ are unmixed, then $\fa\sim _{(I;R)}\fb.$
\end{thm}

\begin{proof}
     There exists a Gorenstain local ring $S$ and an ideal $J$ of $S$ such that $R\cong \frac{S}{J}$. Let $\fa', \fb'$ and $I'$ be ideals of $S$ such that $\fa = \frac{\fa'}{J}$, $\fb = \frac{\fb'}{J}$ and $I = \frac{I'}{J}.$

         By the assumption, $\fa \fb(\frac{w_R}{Iw_R}) = 0$. Also, in view of \cite[3.3.11(c) and 3.3.5(a)]{BH}, $\Ann_{\frac{R}{I}} (\frac{w_R}{Iw_R}) = 0 $. Hence, $\fb \subseteq I:_R \fa$. Therefore, $\fa w_R \subseteq Iw_R :_{w_R}(I:_R\fa) \subseteq Iw_R :_{w_R}\fb = \fa w_R$ and so \begin{equation}\label{e} Iw_R :_{w_R}(I:_R\fa) = Iw_R :_{w_R}\fb .\end{equation}
         Let $\h_R I= t$. We may assume that $\dim R = \dim S$. In other words, there exists an $S$-regular sequence $\fx$ of length $t$ in $I'$. Therefore, by \cite[3.3.7 ]{BH},
\begin{eqnarray*}\frac{w_R}{Iw_R}\cong w_{\frac{R}{I}} \cong \Ext^t_S (\frac{R}{I} , S) \cong \Ext^t_S (\frac{S}{I'} , S) \cong \Hom_{\frac{S}{\fx S}} (\frac{S}{I'} , \frac{S}{\fx S}).\end{eqnarray*}
         Let $I \subseteq \fc$ be an ideal of $R$ and $ \fc'\lhd S$ such that $\fc = \frac{\fc'}{J}$. Then the following natural isomorphisms exist:
\begin{eqnarray*}
0:_{\frac{w_R}{Iw_R}} \fc &\cong&  \Hom_{\frac{R}{I}} (\frac{R}{\fc} , \frac{w_R}{Iw_R})\cong \Hom_{\frac{R}{I}} (\frac{R}{\fc} , \Hom_{\frac{S}{\fx S}} (\frac{S}{I'} , \frac{S}{\fx S})) \\
&\cong& \Hom_{\frac{S}{\fx S}} (\frac{R}{\fc} \bigotimes _{\frac{R}{I}} \frac{R}{I},\frac{S}{\fx S}) \cong \Hom_{\frac{S}{\fx S}} (\frac{R}{\fc} ,\frac{S}{\fx S}) \cong \Hom_{\frac{S}{\fx S}} (\frac{S}{\fc'} ,\frac{S}{\fx S})\cong 0:_{\frac{S}{\fx S}} \fc'.
\end{eqnarray*}
         Therefore, $0:_{\frac{w_R}{Iw_R}} \fb \cong 0:_{\frac{S}{\fx S}} \fb'$ and $0:_{\frac{w_R}{Iw_R}} (I:_R \fa) \cong 0:_{\frac{S}{\fx S}} (I':_S \fa')$. This, in conjunction with (\ref{e}), implies that $0:_{\frac{S}{\fx S}} \fb' \cong 0:_{\frac{S}{\fx S}} (I':_S \fa')$. Also, it is straightforward to see that the combination of these isomorphisms is actually the identity map. In other words, \begin{equation}\label{equ2} 0:_{\frac{S}{\fx S}} \fb' = 0:_{\frac{S}{\fx S}} (I':_S \fa').\end{equation}
         In view of lemma \ref{l3-1} and (\ref{equ2}), $\fb' = I':_S \fa'$ and this, in conjunction with \ref{r2}, shows that $\fb = I:_R \fa.$

          Now the fact that every $w_R$-regular sequence is an $R$-regular sequence, completes the proof.
\end{proof}

In the following theorem, we consider a case where linking over $R$ implies linking over the canonical module $w_R$.

\begin{thm}\label{t3}
Let $I(\subseteq \fa \cap \fb)$ be an ideal of $R$ such that $\fa \neq I$, $\fb\neq I$ and $I$ is
generated by an $R$-regular sequence. Also, assume that $I :_R\fb = \fa$ and $\frac{w_R}{\fb w_R}$ is unmixed.
Then $\fb w_R = Iw_R:_{w_R} \fa$. In other words, if $\fa \sim_{(I;R)} \fb$ and $\frac{w_R}{\fa w_R}$ and $\frac{w_R}{\fb w_R}$
are unmixed, then $\fa \sim _{(I;w_R)} \fb$.
\end{thm}

\begin{proof}

         As $0 :_{\frac{R}{I}}\frac{\fb}{I} = \frac{\fa}{I} \neq 0$ and $I$ is generated by an $R$-regular sequence, $\h_{w_R} \fb = \h_R \fb = \grad_R \fb = \grad_R I = \h_R I = \h_{w_R} I $. Now, using \cite[3.3.5(a)]{BH}, it is straightforward to see that $\frac{w_R}{\fb w_R}$ is unmixed as an $\frac{R}{I}$-module, too.

         Therefore, replacing $R$ with $\frac{R}{I}$, we may assume that $I = 0$. As $\fa \fb = 0 $, $\fb w_R \subseteq 0:_{w_R} \fa$. Assume to the contrary that $0:_{w_R}\fa\neq \fb w_R$. Then there exists $\fp \in \Ass_R \frac{0:{w_R} \fa}{\fb w_R} \subseteq \Ass_R \frac{w_R}{\fb w_R}$.

         Let $E := E_R (\frac{R}{\fp})$ and $\D(-): = \Hom_R(-,E)$. We show that $\fb E = 0:_E \fa$. Since $\fb E \subseteq 0:_E \fa$, there is a natural monomorphim $h : \fb E \rightarrow \D(\frac{R}{\fa})$. Also, using \cite[10.2.16]{BS}, there are the following natural isomorphics
  $$\frac{E}{\fb E} \cong \frac{R}{\fb} \otimes_R \Hom_R(R,E) \cong \Hom_R(\Hom_R(\frac{R}{\fb},R),E) \cong \Hom_R(\fa,E) = \D (\fa).$$

         Now, applying $\D (-)$ to the exact sequence $ 0 \rightarrow \fa \overset{i}\rightarrow R \overset{\pi}\rightarrow \frac{R}{\fa} \rightarrow 0$, we get the following commutative diagrams
         $$\begin{CD}
         &&&&&&&&\\
         \ \ &&&& 0 @>>> \fb E @>>> E @>>> {E}/{\fb E} @>>>0&  \\
          &&&&&&  @VV{h}V   @VV{\cong}V @VV{\cong}V \\
         \ \  &&&& 0 @>>> \D (\frac{R}{\fa}) @>{\D(\pi)}>> \D(R)
 @>{\D(i)}>>\D(\fa) @>>>0&.\\
        \end{CD}$$\\
         This implies that $h$ is an isomorphism. Therefore, the combination $\fb E \overset{h}\rightarrow \D(\frac{R}{\fa}) \rightarrow 0:_E \fa$, which is the inclusion map, is an isomorphism. Hence $\fb E = 0:_E \fa$.

         On the other hand, by the assumption, $\dim \frac{R}{\fp} = \dim \frac{w_R}{\fb w_R} = \dim \frac{R}{\fb} = \dim R$. Therefore, $0 = \dim R_{\fp} = \dim (w_R)_{\fp}$. This implies that, $(w_R)_{\fp} \cong w_{R_{\fp}}\cong E_{R_{\fp}} (\frac{R_{\fp}}{\fp R_{\fp}})$. Therefore, $\fb (w_R)_{\fp} = 0:_{(w_R)_{\fp}} \fa$ and so, $(\frac{ 0:_{w_R}\fa}{\fb w_R})_{\fp} = 0$ which is a contradiction. Hence $\fb w_R = 0:{w_R} \fa$.

    Now, the result follows from the fact that every $R$-regular sequence is an $w_R$-regular sequence too.
\end{proof}


   \section{Some generalizations of a theorem of Peskine and Szpiro}

In this section, we consider some generalizations of the basic result of Peskine and Szpiro \cite[prop 1.3]{PS}, namely if $R$ is Gorenstain, $\fa \neq 0$ and $\fb := 0:_R \fa$ then $\frac{R}{\fa}$ is Cohen-Macaulay if and only if $\frac{R}{\fa}$ is unmixed and $\frac{R}{\fb}$ is Cohen-Macaulay.
\begin{thm}\label{p3}
Let $(R, \fm)$ be a Cohen-Macaulay local ring with the canonical module $w_R$, $I(\subseteq \fa \cap \fb)$ be an ideal of $R$ generating by an $R$-regular sequence such that $\fa \neq I$, $\fb\neq I$ and $I :_R\fa = \fb$. Then the following are equivalent:
\begin{itemize}
     \item  [(i)]  $\frac{w_R}{\fa w_R}$ is Cohen-Macaulay.
     \item  [(ii)] $\frac{w_R}{\fa w_R}$ is unmixed and $\frac{R}{\fb}$ is Cohen-Macaulay.
  \end{itemize}
\end{thm}

\begin{proof}
Considering the fact that $w_\frac{R}{I} \cong \frac{w_R}{I w_R}$, we may replace $R$ by $\frac{R}{I}$ and assume that $I= 0$.

     $"(i) \rightarrow (ii)"$  In view of the Cohen-Macaulayness of $\frac{w_R}{\fa w_R}$, $\frac{w_R}{\fa w_R}$ is unmixed. Consider the exact sequence     \begin{equation}\label{eq2} 0 \rightarrow \fa w_R \rightarrow {w_R} \rightarrow \frac{w_R}{\fa w_R} \rightarrow 0. \end{equation}

      As $\dim \fa w_R \geq \depth \fa w_R \geq \Min \{ \depth w_R , \depth \frac{w_R}{\fa w_R}  + 1\} = \depth R$, $\fa w_R$ is maximal Cohen-Macaulay and so is $\Hom _R(\fa w_R, w_R)$. Now, applying $\Hom _R(- , w_R)$ on (\ref{eq2}) and considering the isomorphisms $\Hom _R(w_R , w_R) \cong R$ and
      \begin{eqnarray*}\Hom _R(\frac{w_R}{\fa w_R} , w_R) &\cong& \Hom _R(\frac{R}{\fa}, \Hom _R(w_R , w_R))\\ &\cong& \Hom_R(\frac{R}{\fa}, R)
       \cong 0:_R \fa = \fb,\end{eqnarray*} we get the following commutative diagrams
     $$\begin{CD}
         &&&&&&&&\\
         \ \ &&&& 0 @>>> \Hom _R(\frac{w_R}{\fa w_R} , w_R) @>>> \Hom _R(w_R , w_R) @>>> \Hom _R(\fa w_R, w_R) @>>>0&  \\
          &&&&&&  @VV{\cong}V   @VV{\cong}V \\
         \ \  &&&& 0 @>>> \fb @>>> R
 @>>>\frac{R}{\fb} @>>>0.&\\
        \end{CD}$$\\

        This implies that $\frac{R}{\fb} \cong \Hom _R(\fa w_R, w_R)$. Hence $\frac{R}{\fb}$ is Cohen-Macaulay.

     $"(ii) \rightarrow (i)"$ Consider the exact sequence \begin{equation}\label{eq3} 0 \rightarrow \fb \rightarrow R \rightarrow \frac{R}{\fb} \rightarrow 0.\end{equation}

     As $\dim \fb \geq \depth \fb \geq \Min \{ \depth R , \depth \frac{R}{\fb}  + 1\} = \depth R$, $\fb$ is maximal Cohen-Macaulay and so is $\Hom _R(\fb , w_R)$. Applying $\Hom _R(- , w_R)$ on (\ref{eq3}) and using \ref{t3} and the natural isomorphism $\Hom _R(\frac{R}{\fb} , w_R) \cong 0:_{w_R} \fb = \fa w_R$, we get the following commutative diagrams
     $$\begin{CD}
         &&&&&&&&\\
         \ \ &&&& 0 @>>> \Hom _R(\frac{R}{\fb} , w_R) @>>> \Hom _R(R , w_R) @>>> \Hom _R(\fb, w_R) @>>>0&  \\
          &&&&&&  @VV{\cong}V   @VV{\cong}V \\
         \ \  &&&& 0 @>>> \fa w_R @>>> w_R
 @>>>\frac{w_R}{\fa w_R} @>>>0.&\\
        \end{CD}$$\\
        This implies that $\frac{w_R}{\fa w_R} \cong \Hom _R(\fb , w_R)$ is Cohen-Macaulay.
\end{proof}

\begin{cor} \label{c8}
Let $(R, \fm)$ be a Cohen-Macaulay local ring with the canonical module $w_R$ and $I$ be an ideal of $R$ such that $\fa\sim _{(I;w_R)}\fb$ and $\fa , \fb$ are pure height. Then the following statements hold.
 \begin{itemize}
   \item [(i)] $\frac{w_R}{\fa w_R}$ is Cohen-Macaulay if and only if $\frac{R}{\fb}$ is Cohen-Macaulay.
   \item [(ii)] $\gd_R \frac{R}{\fa} < \infty$ and $\frac{R}{\fa}$ is Cohen-Macaulay if and only if  $\gd_R \frac{R}{\fb} < \infty$ and $\frac{R}{\fb}$ is Cohen-Macaulay.
   \end{itemize}
\end{cor}

\begin{proof}
Note that, by the assumption and \ref{t2}, $\fa\sim _{(I;R)}\fb$.
  \begin{itemize}
   \item [(i)] It is clear by \ref{p3}.

   \item [(ii)] Note that, $$ \gd_R \frac{R}{\fa} = \Max \{i| \Ext^i_R (\frac{R}{\fa}, R)\neq 0\}.$$ In view of, $$\Ext^i_R (\frac{R}{\fa}, R) \cong \Ext^{i-t}_{\frac{R}{I}} (\frac{R}{\fa}, \frac{R}{I}),$$ where $t := \grad_R\fa$, $\gd_R \frac{R}{\fa} < \infty$ if and only if $\gd_{\frac{R}{I}} \frac{R}{\fa} < \infty$. Hence, we may assume that $I =0$. By the assumption and \cite [1.11]{KHY}, $\frac{w_R}{\fa w_R}$ is Cohen-Macaulay. Hence, by (i), $\frac{R}{\fb}$ is Cohen-Macaulay. On the other hand, $\gd \frac{R}{\fa} = \depth R - \depth \frac{R}{\fa} = ht_R \fa=0$. Therefore, by \cite [1.p 595]{MS}, $\gd \frac{R}{\fb} =0.$
   \end{itemize}
\end{proof}

\begin{cor} \label{c9}
Let $(R, \fm)$ be a Cohen-Macaulay local ring. Assume that $\fa , \fb$ are pure height ideals of $R$ and $\fa\sim _{(0;M)}\fb$ where $M$ is a maximal Cohen-Macaulay $R$-module of finite injective dimension. Then,
\begin{itemize}
   \item [(i)] $\frac{M}{\fa M}$ is Cohen-Macaulay if and only if $\frac{R}{\fb}$ is Cohen-Macaulay.
   \item [(ii)] $\gd_R \frac{R}{\fa} < \infty$ and $\frac{R}{\fa}$ is Cohen-Macaulay if and only if  $\gd_R \frac{R}{\fb} < \infty$ and $\frac{R}{\fb}$ is Cohen-Macaulay.
   \end{itemize}
\end{cor}

\begin{proof}
By \ref{l12}, we may assume that $R$ is complete and it has the canonical module $w_R$. In view of \cite[3.3.28]{BH}, $M \cong \oplus ^l w_R$ for some $l \in\mathbb{N}_0$ and, using \ref{l13}, $\fa\sim _{(0;w_R)}\fb$.

Now, the results follow from \ref{c8}.
\end{proof}

\begin{thm} \label{t5}
Let $(R, \fm)$ be a local ring, $M$ be a Cohen-Macaulay $R$-module of finite injective dimension and $\fa , \fb$ be two ideals of $R$ such that $\fa\sim _{(0;M)}\fb$. Also, assume that $\depth R = \depth\frac{R}{\fa} = \depth \frac{R}{\fb}$. Then, $\frac{M}{\fa M}$ is Cohen-Macaulay if and only if $\frac{M}{\fb M}$ is Cohen-Macaulay.
\end{thm}

\begin{proof}
We prove the claim by induction on $n := \depth R$. If $n=0$ then, by \cite[18.9]{M} and \cite[3.1.23]{BH}, $M$ is injective and $\dim R = 0$. So, the result is clear. Now, assume that $n > 0$. In case $\dim M =0$, there is nothing to prove. Therefore, we may assume that $\dim M > 0$ and there exists $x \in \fm - \Z(M) \cup \Z(R)\cup \Z(\frac{R}{\fa})\cup \Z(\frac{R}{\fb})$. In view of \cite[3.1.24]{BH}, $\Ext^i_R (\frac{R}{\fa},M) = \Ext^i_R (\frac{R}{\fb},M) = 0$ for all $i>0$. Hence, by \ref{l14}, $(\fa,x) \sim _{(0;\frac{M}{xM})}(\fb,x)$. Therefore, by inductive hypothesis, $\frac{M}{(\fa,x) M}$ is Cohen-Macaulay if and only if $\frac{M}{(\fb,x) M}$ is Cohen-Macaulay. Now, the result follows from \ref{R1} and the fact that $x \notin \Z(\frac{M}{\fa M}) \cup  \Z(\frac{M}{\fb M})$.
\end{proof}

\bibliographystyle{amsplain}

\end{document}